\documentclass[12pt, letterpaper]{amsart}
\usepackage{amsfonts,amsthm,mathtools}
\usepackage[all,cmtip]{xy}
\usepackage{fullpage}
\usepackage{amsmath}
\usepackage{amssymb}
\usepackage{enumerate}
\usepackage{tikz}
\usepackage{hyperref}
\usepackage{cleveref}
\usepackage{verbatim}
\usepackage{cancel}
\usepackage{float}
\usepackage{todonotes}

\newtheorem{lem}{Lemma}[section]
\newtheorem{thm}[lem]{Theorem}
\newtheorem{proposition}[lem]{Proposition}
\newtheorem{prop}[lem]{Proposition}
\newtheorem{cor}[lem]{Corollary}
\newtheorem{conj}[lem]{Conjecture}

\theoremstyle{definition}
\newtheorem{remark}[lem]{Remark}
\newtheorem{definition}[lem]{Definition}

\newtheorem{ex}{Example}

\DeclareMathAlphabet{\curly}{U}{rsfs}{m}{n}
\newcommand{\wideunder}{\rule{0.7em}{0.4pt}}

\newcommand{\Id}{\operatorname{Id}}
\newcommand{\End}{\operatorname{End}}
\newcommand{\Hom}{\operatorname{Hom}}

\newcommand{\Sym}{\operatorname{Sym}}
\newcommand{\Pic}{\operatorname{Pic}}

\newcommand{\Q}{\mathbb{Q}}
\newcommand{\C}{\mathbb{C}}

\newcommand{\Z}{\mathbb{Z}}
\newcommand{\F}{\mathbb{F}}
\newcommand{\HH}{\mathbb{H}}

 \newcommand{\lmfdbec}[3]{\href{https://www.lmfdb.org/EllipticCurve/Q/#1/#2/#3}{#1.#2#3}}

  \newcommand{\im}{\operatorname{im}}

\mathchardef\mhyphen="2D

\newcommand{\maarten}[1]{{\color{purple} \sf $\diamondsuit\diamondsuit\diamondsuit$ Maarten: [#1]}}

\title{Modular curves $X_0(N)$ with infinitely many quartic points}

\author{\sc Maarten Derickx}
\address{Maarten Derickx\\
Den Haag, The Netherlands}
\email{maarten@mderickx.nl}
\urladdr{http://www.maartenderickx.nl/}

\author{\sc Petar Orli\'c}
\address{Petar Orli\'c \\
University of Zagreb\\  
Bijeni\v{c}ka Cesta 30 \\
10000 Zagreb\\
Croatia}
\email{petar.orlic@math.hr}

\begin{document}
\begin{abstract}
    We determine all modular curves $X_0(N)$ with infinitely many quartic points. To do this, we define a pairing that induces a quadratic form representing all possible degrees of a rational morphism from $X_0(N)$ to a positive rank elliptic curve.
\end{abstract}


\maketitle

\section{Introduction}
Let $C$ be a smooth, projective, and geometrically integral curve defined over a number field $k$. Determining whether the set of points of degree $\leq d$ on $C$ is finite or infinite is an important problem in arithmetic geometry. Faltings' Theorem solves the base case $d=1$.

\begin{thm}[Faltings' Theorem]
Let $k$ be a number field and let $C$ be a non-singular curve defined over $k$ of genus $g\geq2$. Then the set $C(k)$ is finite.
\end{thm}

Therefore, if $C(k)\neq\emptyset$, then it is infinite if and only if $C$ is isomorphic to $\mathbb{P}^1$ ($g=0$) or $C$ is an elliptic curve ($g=1$) with positive $k$-rank. The next step is considering the same problem for $d>1$. 

\begin{definition}
    Let $C$ be a curve defined over a number field $k$. The arithmetic degree of irrationality $\textup{a.irr}_k C$ is the smallest integer $d$ such that $C$ has infinitely many closed points of degree $d$ over $k$, i.e.
    $$\textup{a.irr}_k C:=\min\left(d, \#\left\{\cup_{[F:k]\leq d} C(F)\right\}=\infty\right).$$
    We also define
    $$\textup{a.irr}_{\overline{k}} \ C:= \min_{[L:k]<\infty} \textup{a.irr}_L \ C.$$
    It is obvious that $\textup{a.irr}_k C\geq \textup{a.irr}_{\overline{k}} \ C$.
\end{definition}

Harris and Silverman \cite{HarrisSilverman91} proved that if a curve of genus $g\geq2$ has infinitely many quadratic points, then it must be either hyperelliptic or bielliptic. Also, Abramovich and Harris \cite{AbramovichHarris91} gave a conjecture
$$\textup{a.irr}_{\overline{k}} \ C\leq d \iff C \textup{ admits a map of degree} \leq d \textup{ to } \mathbb{P}^1 \textup{ or an elliptic curve}$$
which they proved for $d=2,3$. However, Debarre and Fahlaoui constructed counterexamples for $d\geq4$ \cite{DebarreFahlaoui}. One way to characterize when there are infinitely many points of degree $d$ on $C$ is the following theorem.

\begin{thm}[{\cite[Theorem 4.2. (1)]{BELOV}}]
    Let $C$ be a curve over a number field. There are infinitely many degree $d$ points on $C$ if and only if either there exists a map $C\to \mathbb{P}^1$ of degree $d$ or the image of $\Sym^dC$ in $\Pic^d C$ contains a translate of a positive rank abelian variety.
\end{thm}

It can be hard to check, however, whether the image of $\Sym^dC$ in $\Pic^d C$ contains a translate of a positive rank abelian variety. Kadets and Vogt gave a simpler characterization for $d=2,3$, which encompasses the previous results of Harris-Silverman \cite{HarrisSilverman91} and Abramovich-Harris \cite{AbramovichHarris91}. 

\begin{thm}[\cite{KadetsVogt}, Theorem 1.3]\label{kadetsvogt1.3}
    Suppose $X/k$ is a smooth, projective, and geometrically integral curve. Then the following statements hold:
    \begin{enumerate}[(1)]
        \item If $\textup{a.irr}_k X=2$, then $X$ is a double cover of $\mathbb{P}^1$ or an elliptic curve of positive rank over $k$.
        \item If $\textup{a.irr}_k X=3$, then one of the following three cases holds:
        \begin{enumerate}[(a)]
            \item $X$ is a triple cover of $\mathbb{P}^1$ or an elliptic curve of positive rank over $k$.
            \item $X$ is a smooth plane quartic with no rational points, positive rank Jacobian, and at least one cubic point.
            \item $X$ is a genus $4$ Debarre-Fahlaoui curve.
        \end{enumerate}
        \item If $\textup{a.irr}_{\overline{k}} X=d\leq3$, then $X_{\overline{k}}$ is a degree $d$ cover of $\mathbb{P}^1$ or an elliptic curve.
        \item If $\textup{a.irr}_{\overline{k}} X=d=4,5$, then either $X_{\overline{k}}$ is a Debarre-Fahlaoui curve, or $X_{\overline{k}}$ is a degree $d$ cover of $\mathbb{P}^1$ or an elliptic curve.
    \end{enumerate}
\end{thm}

\begin{definition}
    For a curve $C$ defined over a field $k$, the $k$-gonality $\textup{gon}_k C$ is the smallest integer $d$ such that there exists a morphism of degree $d$ from $C$ to $\mathbb{P}^1$ defined over $k$.
\end{definition}

The question of determining $\textup{a.irr}_k C$ is closely related to the $k$-gonality of $C$ and degree $k$ maps to elliptic curves. Frey \cite{frey} proved that if a curve $C$ defined over a number field $k$ has infinitely many points of degree $\leq d$ over $k$, then $\textup{gon}_k C\leq2d$.

Regarding the curves $C=X_1(M,N)$ and $k=\Q$, all cases when $C$ has infinitely many points of degree $d\leq6$ were determined by Mazur \cite{mazur77} (for $d=1$), Kenku, Momose, and Kamienny \cite{KM88, kamienny92} (for $d=2$), Jeon, Kim, and Schweizer \cite{JeonKimSchweizer04} (for $d=3$), Jeon, Kim, and Park \cite{JeonKimPark06} (for $d=4$), and Derickx and Sutherland \cite{DerickxSutherland17} (for $d=5,6$). Additionally, Derickx and van Hoeij \cite{derickxVH} determined all curves $X_1(N)$ which have infinitely many points of degree $d=7,8$ and Jeon determined all trielliptic \cite{Jeon2022} and tetraelliptic \cite{Jeon2023} curves $X_1(N)$ over $\Q$.

In this paper, we will study the curve $C=X_0(N)$ and $k=\Q$. The curve $X_0(N)$ has infinitely many rational points if and only if $N\in\{1-10,12,13,16,18,25\}$ (i.e. when $g(X_0(N))=0$). This was proved by Mazur \cite{mazur78} and Kenku \cite{kenku1979, kenku1980_1, kenku1980_2, kenku1981}. 

Ogg \cite{Ogg74} determined all hyperelliptic curves $X_0(N)$, Bars \cite{Bars99} determined all bielliptic curves $X_0(N)$, as well as all curves $X_0(N)$ with infinitely many quadratic points, and Jeon \cite{Jeon2021} determined all curves $X_0(N)$ with infinitely many cubic points.

\begin{thm}[Bars]\label{thmquadratic}
    The modular curve $X_0(N)$ has infinitely many points of degree $2$ over $\Q$ if and only if
    $$N\in\{1-33,35-37,39-41,43,46-50,53,59,61,65,71,79,83,89,101,131\}.$$
\end{thm} 

\begin{thm}[Jeon]\label{thmcubic}
    The modular curve $X_0(N)$ has infinitely many points of degree $3$ over $\Q$ if and only if
    $$N\in\{1-29,31,32,34,36,37,43,45,49,50,54,64,81\}.$$
\end{thm}

We here determine all curves $X_0(N)$ with infinitely many quartic points. Our main result is the following theorem.

\begin{thm}\label{quarticthm}
    The modular curve $X_0(N)$ has infinitely many points of degree $4$ over $\Q$ if and only if
    \begin{align*}
        N\in\{&1-75,77-83,85-89,91,92,94-96,98-101,103,104,107,111,\\
        &118,119,121,123,125,128,131,141-143,145,155,159,167,191\}.
    \end{align*}
\end{thm}

For $N$ in the above set, we prove in \Cref{infinitelyquarticsection} that $X_0(N)$ has infinitely many quartic points. The harder part of the proof is proving that for the other $N$, there are only finitely many quartic points on $X_0(N)$. 


\begin{remark}
    After submitting this paper to arXiv, we found out that this result has in the meantime been proven in \cite{Hwang2023} using different methods. However, the methods presented here could be used to solve the higher degree cases (i.e. $d=5$). Moreover, with some modifications, our methods could be used to determine all trielliptic and tetraelliptic curves $X_0(N)$ and also all trielliptic quotients of $X_0(N)$ (all bielliptic quotients of $X_0(N)$ were determined in \cite{bars22biellipticquotients}).

    Thus, in this paper we put emphasis on our method of determining the possible degrees of a rational morphism to an elliptic curve. The classification of curves $X_0(N)$ with infinitely many quartic points is given afterwards as an application.
\end{remark}

\Cref{sectionjacobians} contains the technical results used in \Cref{section6}, where we determine all positive rank tetraelliptic curves $X_0(N)$. Our main tool for proving that a curve $C$ over $\Q$ does not admit a degree $4$ morphism to an elliptic curve $E$ over $\Q$ is the representation of rational morphisms from $J_0(N)$ to $E$ by a quadratic form.

\begin{prop}[Part of the proof of \Cref{tetraellthm}]\label{quadraticformprop}
Let $C$ be a curve over $\mathbb Q$ with at least one rational point and $E$ an elliptic curve over $\mathbb Q$ that occurs as an isogeny factor of $J(C)$ with multiplicity $n \geq 1$. Then the degree map $\deg: \Hom_\Q(C,E) \to \Z$ can be extended to a positive definite quadratic form on $\Hom_\Q(J_0(N),E) \cong \Z^n$.
\end{prop}


This statement is a generalization of \cite[Corollary III.6.3]{silverman}, which deals with the case when $C$ is an elliptic curve. The proof uses optimal $E$-isogenous quotients defined in \Cref{section6} to prove that $\Hom_\Q(J(C),E)\cong\Hom_\Q(E^n,E)\cong\Z^n$. We do not give the proof here because we do not need \Cref{quadraticformprop} to prove \Cref{tetraellthm}, but are instead able to manually construct a desired quadratic form in all our cases.

In \Cref{subsectiondegreepairing} we define a pairing on $\Hom_\Q(J(C),E)$ which is an extension of the degree map. \Cref{prop_daggerpairing} tells us that it is a positive definite symmetric bilinear quadratic form. It turns out that, in all our cases, the degeneracy maps $\iota_{d,N,M}$ (defined in \Cref{sectiondegeneracymaps}) form a basis for $\Hom_\Q(J(C),E)$. More precisely, we have the following result.

\begin{prop}
    Let $E$ be an elliptic curve of positive $\Q$-rank and conductor $\textup{Cond}(E)=M\mid N<408$, and let $f:X_0(M)\to E$ be a modular parametrization of $E$. Then (with the natural embedding $X_0(N)\to J_0(N)$), the maps $f\circ \iota_{d,N,M}$ form a basis for $\Hom_\Q(J(C),E)$, where $d$ ranges over all divisors of $\frac{N}{M}$.
\end{prop}

Therefore, the coefficients of the quadratic form are the values of the pairing on the base elements $f\circ \iota_{d,N,M}$. The main result of \Cref{sectionjacobians}, \Cref{pairingcomputation}, allows us to explicitly compute these coefficients in terms of the $q$-expansion of the modular form associated with $E$. We use these quadratic forms in \cref{finitelyquarticsection} to prove the following theorem.

\begin{thm}\label{tetraellthm}
The curve $X_0(N)$ is positive rank tetraelliptic over $\mathbb Q$ if and only if $$N \in \{57, 58, 65, 74, 77, 82, 86, 91, 99, 111, 118, 121, 123, 128, 141, 142, 143, 145, 155, 159\}.$$
\end{thm}

One step of the proof of \Cref{tetraellthm} is to check for finitely many cases that the suitable quadratic form does not take $4$ as a value and conclude that there are no degree $4$ rational morphisms from $X_0(N)$ to $E$. The quadratic forms considered are listed in \Cref{tab:main}. 

The last two sections are an application of our methods developed in \Cref{sectionjacobians} and \Cref{section6}. In \Cref{infinitelyquarticsection}, we find degree $4$ morphisms for all levels $N$ when curve $X_0(N)$ has infinitely many quartic points. 

In \Cref{finitelyquarticsection}, we prove that any curve $C/\Q$ of genus $g\geq8$ with infinitely many quartic points and finitely many cubic points has a degree $4$ morphism to $\mathbb{P}^1$ or a positive rank elliptic curve. Using this result for $C=X_0(N)$, along with the fact that all curves $X_0(N)$ with infinitely many cubic points are tetragonal over $\Q$ (since all of them either have genus $0$ or $1$, or are hyperelliptic or bielliptic), we get that any curve $X_0(N)$ of genus $g\geq8$ with infinitely many quartic points must admit a degree $4$ rational morphism to a positive rank elliptic curve (we will call such curves positive rank tetraelliptic). Therefore, only finitely many levels $N$ (actually, only $N=97$) need to be solved separately.

Finally, at the end of \Cref{finitelyquarticsection}, we prove \Cref{tetraellthm} and \Cref{quarticthm} and give a few examples to illustrate the application of our methods.

The reason why we only solved the case $d=4$ is the following. Although we could get a similar result as in \Cref{tetraellthm} for $d\geq5$, there is a large number of small genus curves $X_0(N)$ for which we cannot use \Cref{kadetsvogt1.4} to connect the degree $d$ maps with the points of degree $\leq d$.

For example, when $d=5$, \Cref{kadetsvogt1.4} can only be used for curves of genus $g\geq12$. Although we can use the Jacobi inversion theorem to deal with the cases $g\leq5$, there are $34$ curves $X_0(N)$ with genus $g\in [6,11]$ such that $J_0(N)$ has positive rank over $\Q$ (for $d=4$ we were lucky to have only one such small genus case, $N=97$). Furthermore, there exists only one pentagonal curve $X_0(N)$, namely $X_0(109)$ \cite[Theorem 1.3]{NajmanOrlic22}, and we did not find any degree $5$ maps from $X_0(N)$ to an elliptic curve. Therefore, we expect that in most $g\geq6$ cases the curve $X_0(N)$ will have only finitely many degree $5$ points.

\section*{Funding}

The second author was supported by the QuantiXLie Centre of Excellence, a project co-financed by the Croatian Government and European Union through the European Regional Development Fund - the Competitiveness and Cohesion Operational Programme (Grant KK.01.1.1.01.0004) and by the Croatian Science Foundation under the project no. IP-2022-10-5008.

\section*{Conflicts of Interest}

The authors have no relevant financial or non-financial interests to disclose.

\section*{Data availability statement}

The Sage codes used in the proofs can be found on:
\url{https://github.com/koffie/mdsage/tree/main/articles/derickx_orlic-quartic_X0}.

\section*{Acknowledgements}

We are grateful to Filip Najman for his helpful comments and Kenneth A. Ribet for providing some very useful references to the literature. We also thank the referees for their comments that have greatly improved the paper.

\section{Jacobians}\label{sectionjacobians}

\subsection{Notation and definitions}\label{subsectiondefinitions}

Let $C,C'$ be curves over a field $k$. A morphism $f: C \to C'$ induces maps $f_* : J(C) \to J(C')$ and $f^* : J(C') \to J(C)$ which are defined as follows. If $D = \sum n_iP_i$ and $D'= \sum n'_iP'_i$ are divisors on $C$ and $C'$ respectively, then $f_*([D]) = [\sum n_if(P_i)]$ and $f^*([D']) = [\sum n'_if^{-1}(P'_i)]$. When seeing $J(C)$ not as divisors modulo principal divisors but as $\Pic^0(C)$, the map $f^*$ is sometimes also denoted as $f^\vee$ or $\Pic(f)$.

\begin{lem}\label{degreepairinglemma}
    $f_* \circ f^* = [\deg f]$.
\end{lem}

\begin{proof}
    $f_*(f^*([D']))=f_*(f^*([\sum n'_iP'_i]))=f_*([\sum n'_if^{-1}(P'_i)])=[\sum n'_i\cdot(\deg f)P'_i]=[(\deg f) D']$.
\end{proof}

By \cite{Milne1986}, Theorem 6.6, the abelian variety $J(C)$ comes with a canonical principal polarization $$\phi_{\Theta_C} : J(C) \to J(C)^\vee$$ induced by the theta divisor of $C$. This map is an isomorphism.

If $P \in C(k)$, then we can define the embedding morphism
\begin{align*}
f_P : C &\to J(C), \\
 x &\mapsto [x - P].
\end{align*}

If $A$ is an abelian variety over $k$, we can use this point to define 
\begin{align*}
\Hom_P(C,A) := \left\lbrace f \in \Hom(C,A) \mid f(P) = 0 \right\rbrace.
\end{align*} 
With this definition, the universal property of the Jacobian (\cite{Milne1986}, Theorem 6.1) states that the map 
\begin{align*}
\iota_P : \Hom(J(C),A) & \to \Hom_P(C,A), \\
h & \mapsto h \circ f_P
\end{align*}
is an isomorphism.

The map 
\begin{align*}
s_P : \Hom(C, A)  &\to \Hom_P(C, A),\\
 f &\mapsto t_{-f(P)} \circ f,
\end{align*}
where $t_{-f(P)}$ denotes the translation by $-f(P)$ map, is a retraction of the canonical inclusion $\Hom_P(C,A) \to \Hom(C,A)$ whose kernel are the constant maps. Since the constant maps can be identified with $A(k)$, we have a direct sum decomposition
\begin{align*}
    \Hom(C,A) \cong \Hom_P(C,A) \times A(k).
\end{align*}
If $A$ is an elliptic curve, then $f$ and $s_P(f)$ have the same degree because $t_{-f(p)}$ is an isomorphism. In particular if one wants to study the possible degrees that occur for elements in $\Hom(C,A)$, it suffices to restrict to those in $\Hom_P(C, A)$.

Note that the maps $f_P^\vee: J(C)^\vee \to J(C)$ and $\phi_{\Theta_C}: J(C) \to J(C)^\vee$ are closely related to each other, namely $f_P^\vee \circ \phi_{\Theta_C} = -\Id_{J(C)}$. For elliptic curves, one often takes $P=0_E$ to be the zero section of the elliptic curve, and then the map $f_{0_E} : E \to J(E)$ is used to identify $E$ with its Jacobian/dual. So the above means that this identification differs by the one coming from the polarization $\phi_{\Theta_E}: J(E) \to J(E)^\vee = E^{\vee\vee} \cong E$ by a minus sign.

\subsection{Degree pairing}\label{subsectiondegreepairing}

We already saw in the previous section that if $f: C \to C'$ is a map of curves, then $f_* \circ f^* = [\deg f]$. This motivates the following definition:
\begin{definition}
Let $C,E$ be curves over a field $k$ with $E$ being an elliptic curve. The degree pairing is defined on $\Hom(C,E)$ as
\begin{align*}
\left<\wideunder,\wideunder\right> : \Hom(C,E)\times\Hom(C,E) & \to \End(J(E))\\
f,g &\mapsto f_*\circ g^*.
\end{align*}
If $P \in C(k)$, then we can define the degree pairing on $\Hom(J(C),E)$ as
\begin{align*}
\left<\wideunder,\wideunder\right> : \Hom(J(C),E)\times\Hom(J(C),E) & \to  \End(J(E)),\\
f,g &\mapsto (f \circ f_P)_*\circ (g \circ f_P)^*.
\end{align*}

\end{definition}

We will also write $\left<f,g\right>:=f_*\circ g^*$ for $f,g\in\Hom(C,C')$ (this is not a pairing when $C'$ is not elliptic since $\Hom(C,C')$ is not an abelian group in that case). With this notation we have $\langle f, f \rangle = [\deg f]$ for $f \in \Hom(C, C')$.

Note that the definition on $\Hom(J(C),E)$ is slightly  unsatisfactory since a priori it seems to depend on the base point $P$. Additionally, it is not defined in terms of intrinsic properties of the abelian variety $J(C)$, but instead just defined by using $f_P: C \to J(C)$ to transport the definition on $\Hom(C,E)$ to that on $\Hom(J(C),E)$. So let's try to give a more intrinsic definition.

Let $A$ and $B$ be two polarized abelian varieties over $k$ with polarizations $\phi_A$ and $\phi_B$ respectively and assume the polarization $\phi_A$ is principal. Then one can define the map
\begin{align*}
\wideunder^\dagger : \Hom(A,B) &\to \Hom(B,A),\\
 f & \mapsto \phi_A^{-1}\circ f^\vee\circ \phi_B.
\end{align*}

When $A=B$ this is just the Rosati involution, defined in Section 17 of \cite{Milne1986AV}.

\begin{definition}
Let $(A,\phi_A)$ and $(B,\phi_B)$ be two polarized abelian varieties over $k$ with $\phi_A$ a principal polarization. Then the dagger pairing on $\Hom(A,B)$ is defined as
\begin{align*}
\left<\wideunder,\wideunder\right>_\dagger : \Hom(A,B)\times\Hom(A,B) & \to  \End(B),\\
f,g &\mapsto f \circ g^\dagger.
\end{align*}
\end{definition}

The following lemma shows how the dagger pairing relates to the degree pairing.
\begin{lem}\label{dagger-degree}
Let $C,E$ be curves over a field $k$ with $E$ being an elliptic curve, and let $P \in C(k)$. Then for $f,g \in \Hom(J(C), J(E))$ we have 
$$\langle f, g\rangle_\dagger = \langle \, f_{0_E}^{-1} \circ f \circ f_P, \,\, f_{0_E}^{-1} \circ g \circ f_P\,\rangle,$$
where the principal polarizations on $J(C)$ and $J(E)$ needed for the definition of $\left<\wideunder,\wideunder\right>_\dagger$ are taken to be those coming from the theta divisors on $C$ and $E$.
\end{lem}
\begin{proof}
We prove this by showing $(f_{0_E}^{-1} \circ f \circ f_P)_* =f$ and $(f_{0_E}^{-1} \circ g \circ f_P)^* = g^\dagger$.

For the equality $(f_{0_E}^{-1} \circ f \circ f_P)_* =f$ it suffices to show equality on points over the algebraic closure of $k$. So let $D = \sum n_i P_i$ be a degree zero divisor representing a point in $J(C)(\overline k)$. Then 
\begin{align*}
&(f_{0_E}^{-1} \circ f \circ f_P)_*(\sum n_i P_i) = \sum n_i (f_{0_E}^{-1} \circ f) (P_i -P) =  (f_{0_E}^{-1} \circ f) (\sum n_i (P_i -P)) =\ldots \\
&\ldots = f_{0_E}^{-1}(f(\sum n_i P_i)-f(\sum n_i P)) =f_{0_E}^{-1}(f(\sum n_i P_i) -f(0)) = f(\sum n_i P_i).
\end{align*}
The equality $(f_{0_E}^{-1} \circ g \circ f_P)^* = g^\dagger$ follows since $^*$ and $^\vee$ denote the same operation and
$$(f_{0_E}^{-1} \circ g \circ f_P)^\vee = f_P^\vee \circ g^\vee \circ (f_{0_E}^{-1})^\vee = (-\phi_{\Theta_C})^{-1}\circ g^\vee \circ (-\phi_{\Theta_E}) =  \phi_{\Theta_C}^{-1}\circ g^\vee \circ \phi_{\Theta_E} = g^\dagger, $$
where the second equality follows by applying Lemma 6.8 of \cite{Milne1986} twice.

\end{proof}

\begin{prop}\label{prop_daggerpairing}
Let $C,E$ be curves over $\Q$ with $E$ being an elliptic curve. Then the dagger pairing is a positive definite symmetric bilinear form on $\Hom_\Q(J(C),J(E))$ taking values in $\End_\Q(J(E))=\mathbb \Z$.
\end{prop}

\begin{proof}
    The dagger pairing is obviously bilinear. It is also symmetric because for $f,g\in\Hom_\Q(J(C),J(E))$ we have $$\left<f,g\right>_\dagger=f\circ g^\dagger=(g\circ f^\dagger)^\dagger=g\circ f^\dagger.$$ Here the last equality holds because $g\circ f^\dagger \in \End_\Q(J(E))$ is of the form $[n]$ for some $n \in \Z$ and $[n]^\dagger=[n]$. The positive definiteness follows from \Cref{dagger-degree} since we can compute $\left<f,f\right>_\dagger$ over $\overline \Q$ by choosing a $P \in C(\overline \Q)$ as follows:
    $$\left<f,f\right>_\dagger=\langle \, f_{0_E}^{-1} \circ f \circ f_P, \,\, f_{0_E}^{-1} \circ f \circ f_P\,\rangle=[\deg f_{0_E}^{-1} \circ f \circ f_P],$$
    and $\deg f_{0_E}^{-1} \circ f \circ f_P > 0$ if $f \neq 0$.
\end{proof}

\begin{remark}
If $E$ is a CM elliptic curve over $\C$, we could also consider the dagger pairing $\Hom_\C(J(C),J(E))\times\Hom_\C(J(C),J(E)) \to \End_\C(J(E))$. This is a positive definite \emph{hermitian} form instead of a symmetric one since the Rosati involution $^\dagger$ acts as complex conjugation on $\End_\C(J(E))$.
\end{remark}

\subsection{Degeneracy maps}\label{sectiondegeneracymaps}

Let $M$ and $N$ be positive integers such that $M\mid N$. For every divisor $d$ of $\frac{N}{M}$ there exists a degeneracy map
$$\iota_{d,N,M}:X_0(N)\to X_0(M), (E,G)\mapsto (E/G[d], (G/G[d])[M]).$$
The degeneracy map acts on $\tau \in \HH^*$ in the extended upper half-plane as
$$\iota_{d,N,M}(\tau)=d\tau.$$
From this or directly from the definition, we can easily see that when $dM \mid N$ and $d'N \mid N'$, then
\begin{align}\iota_{d,N,M}\circ\iota_{d',N',N}=\iota_{dd',N',M}. \label{eq:iota_composition}\end{align}
We want to describe $\left< \iota_{d_1,N,M},\iota_{d_2,N,M}\right>$ for different divisors $d_1$, $d_2$ of $\frac N M$ in terms of Hecke operators on $J_0(M)$ (the case $d_1=d_2$ is solved by \Cref{degreepairinglemma}).
We recall from Section 7.3 of \cite{diamond_im} that Hecke operators $T_n$ act on $Y_0(M)$ as
$$T_n(E,G)=\sum_{\substack{\#C=n \\ C\cap G=\{0\}}} (E/C, (G+C)/C)$$
and have the following properties:
\begin{alignat*}{2}
    T_{p^r}&=T_{p^{r-1}}T_p-pT_{p^{r-2}} &&\textup{ for primes } p\nmid M\textup{ and } r>1,\\
    T_{p^r}&=T_p^r &&\textup{ for primes } p\mid M\textup{ and } r>0,\\
    T_{mn}&=T_mT_n &&\textup{ if } (m,n)=1.
\end{alignat*}

We want to determine $\left<\iota_{d_1,N,M},\iota_{d_2,N,M}\right>$. When $N=Mp$ for a prime $p$, we already know from Section 7.3 of \cite{diamond_im} that $\left< \iota_{1,N,M},\iota_{p,N,M}\right>=T_p$. The remaining case is when $\frac{N}{M}$ is a composite number. Before we consider that case, we prove a technical group theory lemma.

\begin{lem}\label{grouplemma}
    Let $G$ be an abelian group of order $N$ that has a cyclic subgroup $G'$ of order $d$. If $dG\cong \Z/(N/d)\Z$, then $G$ is cyclic.
\end{lem}

\begin{proof}
    We know that $G\cong(\Z/(d_1\Z))\times\ldots\times(\Z/(d_k\Z))$, where $d_i$ are integers such that $d_1\mid\ldots\mid d_k$, $d_1\ldots d_k=N$ and $d\mid d_k$. Thus,
$$dG\leq(\Z/(d_1\Z))\times\ldots\times\left(\Z/\left(\frac{d_k}{d}\Z\right)\right).$$
    However, $dG\cong \Z/\left(\frac{N}{d}\Z\right)$ implying that $d_1=\ldots=d_{k-1}=1$ and $d_k=N$. 
\end{proof}

\begin{lem}\label{lem1}
    If $\frac{N}{M}$ is square-free, then $\left<\iota_{1,N,M},\iota_{N/M,N,M}\right>=T_{N/M}$. 
\end{lem}

\begin{proof}
    Suppose that $(E,G)$ represents a point on $Y_0(M)$. We compute
    \begin{align*}
        \left<\iota_{1,N,M},\iota_{N/M,N,M}\right>(E,G)&=\sum_{\substack{E'/G'[N/M]=E \\ G'/G'[N/M]=G}} (E',G'[M]),\\
        T_{N/M}(E,G)&=\sum_{\substack{\#C=N/M \\ C\cap G=\{0\}}} (E/C,(G+C)/C)
    \end{align*}
    In order to prove that these sums are equal, it is enough to find a bijection between the summands. We will now construct a map that sends $(E',G'[M])$ to $(E/C,(G+C)/C)$ (i.e. define $C$ in terms of $E'$ and $G'$) and prove that it is a bijection.

    By definition, there is a map $f:E'\to E$ such that $\ker f=G'[N/M]$. We set
    $$C:=\ker f^\vee.$$
    This means that for the map $f^\vee:E\to E'$ we have $E'=E/C$. Further, 
    $$(G+C)/C=f^\vee(G)=f^\vee(f(G'))=\frac{N}{M}G'=G'[M],$$
    meaning that $G\cap C=\{0\}$ (since $f^\vee(G)$ is a group of order $M$) and $(E/C,(G+C)/C=(E',G'[M]))$.
    To prove bijectivity, we define the inverse map, i.e. we define $E'$ and $G'$ in terms of $C$.

    By definition, there is a map $g:E\to E/C$. We set
    \begin{align*}
        E'&:=E/C,\\
        G'&:=(g^\vee)^{-1}(G).
    \end{align*}
    First, we need to prove that $G'$ is a cyclic subgroup of order $N$. It is obviously a group of order $N$. We have
    $$\frac{N}{M}G'=(g\circ g^\vee)(g^\vee)^{-1}(G)=g(G)=(G+C)/C.$$
    Also, $\Z/\left(\frac{N}{M}\Z\right)\cong \ker g^\vee\leq G'$ so we can use \Cref{grouplemma} to conclude that $G'$ is a cyclic subgroup of order $N$. This further implies that $G'[M]=\frac{N}{M}G'=(G+C)/C$.

    To prove that these two maps are inverse to each other it is enough to prove $g=f^\vee$. This holds because
    $$\ker f=G'[N/M]=(g^\vee)^{-1}(G)[N/M]=M(g^\vee)^{-1}(G)=(g^\vee)^{-1}(MG)=(g^\vee)^{-1}(0)=\ker g^\vee.$$
\end{proof}

\begin{lem}\label{lem2}
    If $\frac{N}{M}$ is square-free, then $\left<\iota_{N/M,N,M},\iota_{1,N,M}\right>=w_M\circ T_{N/M}\circ w_M$.
\end{lem}

\begin{proof}
We will prove the equivalent statement $w_M\circ \left<\iota_{N/M,N,M},\iota_{1,N,M}\right>=T_{N/M}\circ w_M$. Suppose that $(E,G)$ represents a point on $Y_0(M)$. We compute (similarly as in the previous lemma)
\begin{align*}
    \left<\iota_{N/M,N,M},\iota_{1,N,M}\right>(E,G)&=\sum_{\substack{\#G'=N \\ G'[M]=G}} (E/G'[N/M],G'/G'[N/M])=\sum (E',G''),\\
    w_M\circ \left<\iota_{N/M,N,M},\iota_{1,N,M}\right>(E,G) &=\sum_{\substack{\#G'=N \\ G'[M]=G}} (E'/G'', E'[M]/G''),\\
    T_{N/M}\circ w_M(E,G)&=\sum_{\substack{\#H=N/M \\ (E[M]/G)\cap H=\{0\}}} (E/G/H, ((E[M]/G)+H)/H).
\end{align*}
It remains to prove that there is a bijection between the summands. We have the following situation:
$$E\mathop{\rightarrow}^{f_1} E'\mathop{\rightarrow}^{f_2}E'/G'',$$
$$E\mathop{\rightarrow}^{g_1}E/G\mathop{\rightarrow}^{g_2}E/G/H$$
where we know that $G'=\ker (f_2\circ f_1)$ because $\ker f_2=G'/G'[N/M]=G'/\ker f_1$.

We can express $G'$ in terms of $H$ as $G':=g_1^{-1}(H)=\ker (g_2\circ g_1)$. By \Cref{grouplemma}, this is a cyclic group of order $N$ because $\Z/(M\Z)\cong G\leq G'$ and
$$MG'=(g_1^\vee\circ g_1)g_1^{-1}(H)=g_1^\vee(H)\cong H\cong \Z/\left(\frac{N}{M}\Z\right).$$
Here the third equality holds because $(E[M]/G)\cap H=\{0\}$. Now, since $G'=g_1^{-1}(H)=\ker (g_2\circ g_1)$, we get $f_2\circ f_1=g_2\circ g_1$. Further, $G=\ker g_1\subset G'$ implying that $G'[M]=G$.

Let us now express $H$ in terms of $G'$. Since $G$ is a subgroup of $G'=\ker (f_2\circ f_1)$, there exist isogenies $g_1$ and $g_2$ such that $g_2\circ g_1=f_2\circ f_1$ and $G=\ker g_1$. We set $H:=\ker g_2$. It remains to prove that $(E[M]/G)\cap H=\{0\}$. This holds because
$$g_2(E[M]/G)=g_2(g_1(E[M]))=f_2(f_1(E[M]))=E[M]/(E[M]\cap G')\cong E[M]/G.$$
\end{proof}

\begin{remark}\label{squarefreeremark1}
    When $\frac{N}{M}$ is not square-free, then the proof of \Cref{lem1} and \Cref{lem2} still work provided one replaces the Hecke operator $T_{N/M}$ by a slightly different operator $T'_{N/M}$, which is defined as
    $$T'_n(E,G):=\sum_{\substack{\#C=n \\ C \textup{ cyclic}\\ C\cap G=\{0\}}} (E/C, (G+C)/C).$$ Note that the only difference between $T'$ and $T$ is that the sum in $T'$ is restricted to cyclic subgroups. If $N/M$ is coprime to $M$, then $T_{N/M}$, seen as element of $\End (J_0(M))$, can be easily expressed as $T_{N/M} = \sum_{m^2\mid N/M} \mu(m)T'_{N/(Mm^2)}$. Using the M\"obius inversion formula one then gets
    $$\left<\iota_{1,N,M},\iota_{N/M,N,M}\right>=T'_{N/M} =\sum_{m^2\mid N/M} \mu(m)T_{N/(Mm^2)},$$
    $$\left<\iota_{N/M,N,M},\iota_{1,N,M}\right>=w_M \circ T'_{N/M} \circ w_M =w_M \circ \sum_{m^2\mid N/M} \mu(m)T_{N/(Mm^2)} \circ w_M,$$
    where $\mu$ denotes the M\"obius function.
\end{remark}

\begin{proposition}\label{iotapairingcomposition}
Let $M,d_1,d_2$ be positive integers with $\gcd(d_1,d_2)=1$. Then
\begin{align*}\iota_{1,d_1d_2M,d_1M,*} \circ \iota_{1,d_1d_2M,d_2M}^* &= \iota_{1,d_1M,M}^* \circ \iota_{1,d_2M,M,*} \text{\quad and}\\
\left<\iota_{d_1,d_1d_2M,M},\iota_{d_2,d_1d_2M,M}\right> &= \left<\iota_{d_1,d_1M,M},\iota_{1,d_1M,M}\right> \circ \left<\iota_{1,d_2M,M},\iota_{d_2,d_2M,M}\right>.
\end{align*}
\end{proposition}

\begin{proof}
Let $E$ be an elliptic curve with a cyclic subgroup $G$ of order $d_2M$. The first equality can be verified on a pair $(E,G)$ since 
\begin{align*}
\iota_{1,d_1d_2M,d_1M,*} \circ \iota_{1,d_1d_2M,d_2M}^*(E,G) &= \sum_{\substack{H_1 \supseteq G \textup{ cyclic} \\ \#H_1 = d_1d_2M}} (E,H_1[d_1M]) \\
&= \sum_{\substack{H_2 \supseteq G[M]\textup{ cyclic} \\ \#H_2 = d_1M}} (E,H_2) \\
&= \iota_{1,d_1M,M}^* \circ \iota_{1,d_2M,M,*}(E,G).
\end{align*}
Furthermore, $H_1$ and $H_2$ are related to each other via $H_2=H_1[d_1M]$ and $H_1=H_2+G$.
The second equality follows from the first because
\begin{align*}
    \left<\iota_{d_1,d_1d_2M,M},\iota_{d_2,d_1d_2M,M}\right> 
    &= \iota_{d_1,d_1d_2M,M,*} \circ \iota_{d_2,d_1d_2M,M}^*\\
    &=\iota_{d_1,d_1M,M,*}\circ\iota_{1,d_1d_2M,d_1M,*} \circ \iota_{1,d_1d_2M,d_2M}^* \circ \iota_{d_2,d_2M,M}^* \\
    &= \iota_{d_1,d_1M,M,*}\circ\iota_{1,d_1M,M}^* \circ \iota_{1,d_2M,M,*} \circ \iota_{d_2,d_2M,M}^*\\
    &=\left<\iota_{d_1,d_1M,M},\iota_{1,d_1M,M}\right>\circ \left<\iota_{1,d_2M,M},\iota_{d_2,d_2M,M}\right>.
\end{align*}
\end{proof}

Combining the previous results we get the following proposition.

\begin{prop}\label{prop:iota_pairing}
    Assume that $\frac{N}{M}$ is either squarefree or coprime to $M$ and let $d_1$ and $d_2$ be divisors of $\frac{N}{M}$. We write just $\gcd$ instead of $\gcd(d_1,d_2)$ and $\textup{lcm}$ instead of $\textup{lcm}(d_1,d_2)$ for simplicity. Then
    \begin{align*}
        \left<\iota_{d_1,N,M},\iota_{d_2,N,M}\right>=&\ w_M\circ  \left(\sum_{m^2|d_1/\gcd} \mu(m)T_{d_1/(m^2\gcd)}\right)\circ w_M\circ\\
        &\circ \left(\sum_{m^2|d_1/\gcd} \mu(m)T_{d_2/(m^2\gcd)}\right)\circ[\deg\iota_{\gcd,N,M\textup{lcm}/\gcd}].
    \end{align*}
\end{prop}

\begin{proof}
 Note that $$\iota_{d_1,N,M} = \iota_{d_1/\gcd,M\textup{lcm}/\gcd,M} \circ \iota_{\gcd,N,M\textup{lcm}/\gcd}$$
and similarly for $d_2$. This shows that $\iota_{d_1,N,M}$ and $\iota_{d_2,N,M}$ both factor through the map $\iota_{\gcd,N,M\textup{lcm}/\gcd}$ allowing us to write
\begin{align*}
    \left<\iota_{d_1,N,M},\iota_{d_2,N,M}\right>&=\iota_{d_1/\gcd,M\textup{lcm}/\gcd,M,*}\circ\iota_{\gcd,N,M\textup{lcm}/\gcd,*}\circ\iota_{\gcd,N,M\textup{lcm}/\gcd}^*\circ\iota_{d_2/\gcd,M\textup{lcm}/\gcd,M}^*\\
    &=\iota_{d_1/\gcd.M\textup{lcm}/\gcd,M,*}\circ[\deg\iota_{\gcd,N,M\textup{lcm}/\gcd}]\circ\iota_{d_2/\gcd,M\textup{lcm}/\gcd,M}^*\\
    &=\left<\iota_{d_1/\gcd,M\textup{lcm}/\gcd,M},\iota_{d_2/\gcd,M\textup{lcm}/\gcd,M}\right>\circ[\deg\iota_{\gcd,N,M\textup{lcm}/\gcd}].
\end{align*}
    
    Further, by \Cref{iotapairingcomposition} we have   $$\left<\iota_{d_1/\gcd,M\textup{lcm}/\gcd,M},\iota_{d_2/\gcd,M\textup{lcm}/\gcd,M}\right>=\left<\iota_{d_1/\gcd,Md_1/\gcd,M},\iota_{1,Md_1/\gcd,M}\right>\circ\left<\iota_{1,Md_2/\gcd,M},\iota_{d_2/\gcd, Md_2/\gcd, M}\right>.$$
    Now we get the desired result by applying \Cref{lem1}, \Cref{lem2}, and \Cref{squarefreeremark1}.
\end{proof}

\begin{thm}\label{pairingcomputation}
Using the assumptions and notation of \Cref{prop:iota_pairing}, let $E$ be an elliptic curve of conductor $M$ with corresponding newform $\sum_{n=1}^\infty a_nq^n$ and let $f:X_0(M)\to E$ be the modular parametrization of $E$. If we define
$$a=\left(\sum_{m^2\mid (d_1/\gcd)} \mu(m)a_{d_1/(\gcd m^2)}\right)\left(\sum_{m^2\mid (d_2/\gcd)} \mu(m)a_{d_2/(\gcd m^2)}\right),$$
where $\mu$ is the M\"obius function (when $\frac{d_1d_2}{\gcd^2}$ is squarefree, $a$ is equal to $a_{d_1d_2/\gcd^2}$, similarly as in \Cref{squarefreeremark1}), then
$$\left<f\circ\iota_{d_1,N,M},f\circ\iota_{d_2,N,M}\right>=\left[a\cdot \frac{\psi\left(N\right)}{\psi\left(\frac{M\textup{lcm}}{\gcd}\right)}\cdot \deg f \right].$$
Here $\psi(N)=N\prod_{q\mid N}(1+\frac{1}{q})$, as in \Cref{oggineq1}.
\end{thm}

\begin{proof}
    For the sake of simplicity, we will assume that $\frac{d_1d_2}{\gcd^2}$ is squarefree. We have
    $$\left<f\circ\iota_{d_1,N,M},f\circ\iota_{d_2,N,M}\right>=f_*\circ \iota_{d_1,N,M,*}\circ\iota_{d_2,N,M}^*\circ f^*=f_*\circ\left<\iota_{d_1,N,M},\iota_{d_2,N,M}\right>\circ f^*.$$
    Let $E'$ be $f^*(E)\subset J_0(M)$. Then $E'$ is an elliptic curve isogenous to $E$. Since, up to isogeny, $E$ occurs with multiplicity one in the factorization of $J_0(M)$ (because $\textup{cond}(E)=M$), it follows that $\left<\iota_{d_1,N,M},\iota_{d_2,N,M}\right>$ is a rational endomorphism of $E'$. Therefore, $\left<\iota_{d_1,N,M},\iota_{d_2,N,M}\right>$ is of the form $[k]$ for some $k\in\Z$ and we get that
    $$f_*\circ\left<\iota_{d_1,N,M},\iota_{d_2,N,M}\right>\circ f^*=f_*\circ f^*\circ\left<\iota_{d_1,N,M},\iota_{d_2,N,M}\right>=[\deg f]\circ\left<\iota_{d_1,N,M},\iota_{d_2,N,M}\right>.$$
    \Cref{prop:iota_pairing} now tells us that
    \begin{align}\label{eq1}
    \left<f\circ\iota_{d_1,N,M},f\circ\iota_{d_2,N,M}\right>=w_M\circ T_{d_1/\gcd}\circ w_M\circ T_{d_2/\gcd}\circ [\deg\iota_{\gcd,N,M\textup{lcm}/\gcd}]\circ [\deg f].
    \end{align}
    We see that here both the Atkin-Lehner involution $w_M$ and Hecke operators act on $E$. As $w_M$ acts as $\pm1$ on $E$, the action of $w_M$ cancels itself. Furthermore, the Hecke operators $T_n$ act on $E$ as multiplication by $a_n$ (the coefficient in the corresponding newform). Therefore,
    \begin{align*}
        \left<f\circ\iota_{d_1,N,M},f\circ\iota_{d_2,N,M}\right>&= [a_{d_1/\gcd}]\circ [a_{d_2/\gcd}]\circ [\deg\iota_{\gcd,N,M\textup{lcm}/\gcd}]\circ [\deg f]\\
        &=[a_{d_1d_2/\gcd^2}\cdot\deg\iota_{\gcd,N,M\textup{lcm}/\gcd}\cdot\deg f].
        \end{align*}
    The last equality holds due to the fact that $a_na_m=a_{nm}$ for relatively prime $m,n$. Finally, since the degrees of all degeneracy maps from $X_0(N)$ to $X_0(M\textup{lcm}/\gcd)$ are equal to $\frac{\psi\left(N\right)}{\psi\left(\frac{M\textup{lcm}}{\gcd}\right)}$, we get the desired formula.

    If $\frac{d_1d_2}{\gcd^2}$ is not squarefree, in \ref{eq1} we will get the M\"obius sums from \Cref{squarefreeremark1} instead of $T_{d_i/\gcd}$. We can then use the same argument to get the desired result since the sums $\sum_{m^2\mid (d_i/\gcd)} \mu(m)T_{d_i/(\gcd m^2)}$ act on $E$ as $\sum_{m^2\mid (d_i/\gcd)} \mu(m)a_{d_i/(\gcd m^2)}$.
\end{proof}


This result is useful because all items on the right-hand side are easily computable ($\deg f$ is the modular degree of $E$ and $a$ is determined by the coefficients of the corresponding newform of $E$), and in fact already have been computed for all elliptic curves of conductor $\leq500,000$ and $\textup{lcm}(d_1,d_2)/\gcd(d_1,d_2) \leq 1,000$. This data is available in the LMFDB \cite{lmfdb}.

\begin{remark}\label{pairingcomputationmagma}
Alternatively, we can compute $\left<f\circ\iota_{d_1,N,M},f\circ\iota_{d_2,N,M}\right>$ using either Sage or Magma since $\iota_{d_1,N,M,*}$ and $\iota_{d_2,N,M}^*$ are explicitly computable on modular symbols, see Proposition 8.26 of \cite{stein07}.
\end{remark}

\section{\texorpdfstring{$d$}{d}-elliptic modular curves}\label{section6}

\begin{definition}
Let $d$ be a positive integer. We call a curve $C$ over a field $k$ $d$-elliptic if there exists an elliptic curve $E$ over $k$ and a morphism $C \to E$ of degree $d$ defined over $k$. If in addition $k$ is a number field and $E$ has positive Mordell-Weil rank, then we call $C$ positive rank $d$-elliptic.
\end{definition}

In this section, we will describe some ideas that allow one to determine for given integers $N$ and $d$ whether $X_0(N)$ is $d$-elliptic over $\mathbb Q$.

If we fix a point $P\in X_0(N)(\Q)$, then, as we have seen in \Cref{subsectiondefinitions}, there exists an element of $\textup{Hom}_\Q(X_0(N),E)$ of degree $d$ if and only if there exists an element of $\textup{Hom}_{\Q,P}(X_0(N),E)$ of degree $d$. Furthermore, by the universal property of $J_0(N)$, every $f\in\textup{Hom}_{\Q,P}(X_0(N),E)$ factors uniquely through $J_0(N)$ via the map $f_P$.

We define a map $\textup{Hom}_\Q(X_0(N),E)\to\textup{Hom}_\Q(J_0(N),E)$ as follows:
$$f\mapsto t_{-f(P)}\circ f\mapsto \textup{homomorphism induced from } t_{-f(P)}\circ f\textup{ by the universal property of } J_0(N).$$
In this section, to make the text more readable, we sometimes use a slight abuse of notation. We will sometimes work with maps defined on $X_0(N)$ as if they were defined on $J_0(N)$. For example, in the proof of \Cref{tetraellthm}, we will say that the maps $f\circ d_i:X_0(N)\to E$ form a basis for $\textup{Hom}_\Q(J_0(N),E)$, but this will actually hold for the images of $f\circ d_i$ via the above map. 

\begin{definition}
Let $A$ and $B$ be abelian varieties over a field $k$ with $B$ simple. An abelian variety $A'$ together with a quotient map $\pi: A \to A'$ is an optimal $B$-isogenous quotient if $A'$ is isogenous to $B^n$ for some integer $n$ and every morphism $A \to B'$ with $B'$ isogenous to $B^m$ for some integer $m$ uniquely factors via $\pi$.
\end{definition}

\begin{prop}
Optimal $B$-isogenous quotients exist, and are unique up to a unique isomorphism.
\end{prop}
\begin{proof}
By the Poincar\'e reducibility theorem (\cite[Chapter 10, Proposition 10.1]{Milne1986AV} or \cite[Theorem 5.3.7]{Birkenhake2004}), there exists an integer $s$ and simple abelian subvarieties $A_1,\ldots, A_s$ of $A$ such that the sum map $A_1 \times \cdots \times A_s \to A$ is an isogeny. By reordering the $A_i$ if necessary we can let $n \leq s$ be the integer such that $A_1,\ldots, A_n$ are isogenous to $B$ while $A_{n+1}, \ldots, A_s$ are not. Define $A' = A/(A_{n+1}+\cdots+A_s)$ then $A'$ is isogenous to $B^n$ since the composition of the maps $A_1 \times \cdots \times A_n \to A \to A'$ is an isogeny.

To show that the quotient $\pi: A \to A'$ is an optimal $B$-isogenous quotient, let $B'$ be an abelian variety isogenous to $B^m$ and let $f: A \to B'$ be a morphism. Since $B'$ is isogenous to $B^m$ but all the $A_i$ for $i>n$ are not isogenous to $B$, meaning that for $i>n$, $A_i \subset \ker f$. However, $A'$ was obtained by quotienting out the $A_i$ with $i>n$ meaning that $f$ factors uniquely via $\pi$ which is what we needed to prove.

The uniqueness up to unique isomorphism follows formally because optimal $B$-isogenous quotients are defined using a universal property.
\end{proof}

\begin{remark}When $E$ is the strong Weil curve over $\mathbb Q$ of conductor $M$, then the optimal $E$-isogenous quotient of $J_0(M)$ is just the strong Weil parameterization of $E$.
\end{remark}

The dual notion of optimal $B$-isogenous quotient is the following:
\begin{definition}
Let $A$ and $B$ be abelian varieties over a field $k$ with $B$ simple. An abelian variety $A'$ together with an isogeny $\iota: A' \to A$ is a maximal $B$-isogenous subvariety if $A'$ is isogenous to $B^n$ for some integer $n$ and every morphism $B' \to A$ with $B'$ isogenous to $B^m$ for some integer $m$ uniquely factors via $\iota$.
\end{definition}
The following follows formally from duality since we can just take $\iota = \pi^\vee$ where $\pi$ is an optimal $B$-isogenous quotient of $A^\vee$. The reason for calling $A'$ a subvariety is because, by the universal property, $\iota$ actually induces an isomorphism between $A'$ and $\iota(A')$.
\begin{prop}
Maximal $B$-isogenous subvarieties exist, and are unique up to a unique isomorphism.
\end{prop}
\begin{remark}\label{remark:maximal_B_isogenous}
The above proposition can also be proved constructively. Namely, if $A_1,\ldots, A_s$ are simple abelian subvarieties of $A$ such that the sum map $A_1 \times \cdots \times A_s \to A$ is an isogeny and additionally $A_1,\ldots, A_n$ are isogenous to $B$ while $A_{n+1}, \ldots, A_s$ are not. Then $A_1+\cdots+A_n \subseteq A$ is a maximal $B$-isogenous subvariety.
\end{remark}

\begin{definition}\label{degeneracybasisdef}
    Let $N $ and $M$ be positive integers with $M \mid N$ and let $n$ denote the number of divisors of $N/M$. Then we define the maps $\tau_{N,M}: J_0(N)\to J_0(M)^n$, $\tau^*_{N,M}: J_0(M)^n\to J_0(N)$ as 
    \begin{align*}
        \tau_{N,M}&:=(\iota_{1,N,M,*},\ldots,\iota_{N/M,N,M,*}),\\
        \tau^*_{N,M}&:=(\iota^*_{1,N,M},\ldots,\iota^*_{N/M,N,M}),
    \end{align*}
    where the first subscript of $\iota$ runs over all divisors of $N/M$.
    Further, let $A$ be an abelian variety and $f: J_0(M) \to A$ a morphism. Then we define the map $\xi_{f,N}:J_0(N)\to A^n$ as 
    $$\xi_{f,N}:= f^n \circ \tau_{N,M}.$$
    If $A$ is a strong Weil curve $E$ of conductor $M$ and $f$ is its modular parametrization, then we use the notation $\xi_{E,N}:=\xi_{f,N}$.
\end{definition}

With the above notation we have $\tau_{N,M} = \xi_{id_{J_0(M)},N}$.

\begin{prop}\label{optimalquotient}
Suppose $N < 408$ and let $E$ be a strong Weil curve over $\Q$ of positive rank and conductor $\textup{Cond}(E)=M\mid N$. If $n$ is the number of divisors of $N/M$, then $\xi_{E,N}^\vee: E^n \to J_0(N)$ has a trivial kernel. Hence $\xi_{E,N}^\vee: E^n \to J_0(N)$ is a maximal $E$-isogenous abelian subvariety and $\xi_{E,N}: J_0(N) \to E^n$ is an optimal $E$-isogenous quotient of $J_0(N)$.
\end{prop}
\begin{proof}
The claim that $\xi_{E,N}^\vee: E^n \to J_0(N)$ is injective for $N < 408$  was verified computationally using Sage. It is a finite computation since the restriction on $N$ means there are only finitely pairs $(N,E)$ for which we need to verify that $\xi_{E,N}^\vee: E^n \to J_0(N)$ is injective. 

The second part follows from Atkin-Lehner-Li Theory.
The decomposition $$S_2(\Gamma_0(N)) = \bigoplus_{M \mid N} \bigoplus_{d \mid N/M} \iota_{d,N,M}^*(S_2(\Gamma_0(M))_{\textup{new}})$$ from Theorem 9.4 of \cite{stein07} yields the isogeny decomposition 
$$J_0(N) = \bigoplus_{M \mid N} \bigoplus_{d \mid N/M} \iota_{d,N,M}^*(J_0(M)_{\textup{new}}).$$ If $E/\mathbb Q$ is an elliptic curve of conductor $M$, then $M$ is the only integer such that $E$ occurs as an isogeny factor of $J_0(M)_{\textup{new}}$, and does so with multiplicity one. In particular, this decomposition implies that if $E$ is an elliptic curve of conductor $M$ and $f: J_0(M) \to E$ is its modular parametrization, then the maps $(f\circ\iota_{d,N,M})^\vee: E \to J_0(N)$ give all the isogeny factors of $J_0(N)$ that are isogenous to $E$, where $d$ ranges over all divisors of $N/M$. From \Cref{remark:maximal_B_isogenous} it then follows that the image of $\xi_{E,N}^\vee$ inside $J_0(N)$ is a maximal $E$-isogenous subvariety of $J_0(N)$. However, since we already verified that $\xi_{E,N}^\vee$ has a trivial kernel, we have that $\xi_{E,N}^\vee$ is an isomorphism onto its image. In particular, $\xi_{E,N}^\vee$ is also a maximal $E$-isogenous subvariety of $J_0(N)$.
\end{proof}

The fact that the map $\xi_{E,N}^\vee$ has a trivial kernel for the cases in which the above proposition is applicable makes it significantly easier to determine the positive rank tetraelliptic $X_0(N)$.
All elliptic curves of positive $\Q$-rank and conductor at most $408$ have rank $1$, with the exception of the curve $389.a1$ which has rank $2$. The following proposition is not needed for the classification of the positive rank tetraelliptic curves $X_0(N)$ in \Cref{tetraellthm}. Instead, it is an attempt to explain why we observed that the kernel of $\xi_{E,N}^\vee$ was always trivial in \Cref{optimalquotient}.

\begin{prop}\label{2groupkernel}
Let $E$ be a strong Weil curve over $\Q$ of conductor $M\mid N$ and let us suppose that $\frac{N}{M}$ is squarefree and coprime to $M$. If $E$ has an odd analytic rank, then the kernel of $\xi_{E,N}^\vee: E^n \to J_0(N)$ is a $2$-group ($n$ is again the number of divisors of $N/M$).
\end{prop}

The main ingredient in the proof of this proposition is \Cref{thmling}.
\begin{definition}
    Let $M$ be a positive integer and let $\pi:X_1(M)\to X_0(M)$ be the natural map $(E,P)\mapsto(E,\left<P\right>)$. The Shimura subgroup $\Sigma(M)$ is the kernel of the map $\pi^*:J_0(M)\to J_1(M)$. For an abelian subvariety $A \subseteq J_0(M)$ we define the Shimura subgroup of $A$ to be $A \cap \Sigma(M)$
\end{definition}

\begin{thm}[Theorem 4 from \cite{LING199539}]\label{thmling}
    Let $N$ be a positive integer, and let $M$ be a divisor of $N$ such that $\frac{N}{M}=q_1\ldots q_t$ (distinct primes) and $\left(M,\frac{N}{M}\right)=1$. We define
    $$\Sigma(M)_0^{2^t}:=\left\{(x_1,\ldots,x_{2^t}): x_i\in \Sigma(M), \sum_1^{2^t} x_i=0\right\}.$$
    We recall from \Cref{degeneracybasisdef} the map $\tau^*_{N,M}:=(\iota^*_{1,N,M},\ldots,\iota^*_{N/M,N,M}): J_0(M)^{2^t}\to J_0(N)$.
    \begin{enumerate}[(i)]
        \item If $M$ is odd or $\frac{N}{M}$ is a prime, then $\ker \tau_{N,M}^* = \Sigma(M)_0^{2^t}$.
        \item If $M$ is even and $\frac{N}{M}$ is not a prime, then $\ker \tau_{N,M}^*$ and $\Sigma(M)_0^{2^t}$ are equal up to a $2$-group.
    \end{enumerate}
\end{thm}

\begin{proof}[Proof of \Cref{2groupkernel}]
\Cref{thmling} tells us that the kernel of $\tau_{N,M}^*$ is equal to $\Sigma(M)_0^{2^t}$ up to a $2$-group. Since $E$ is a strong Weil curve, we have that $f^\vee: E \to J_0(M)$ actually turns $E$ into a subvariety of $J_0(M)$. Therefore, we have $\ker \xi_{E,N}^\vee=\ker \tau_{N,M}^* \cap E^{2^t}$. This is, up to a $2$-group, equal to $\Sigma(M)_0^{2^t}\cap E^{2^t}$, which is isomorphic to $(\Sigma(M)\cap E)^{2^t-1}$. Now it is enough to prove that $\Sigma(M)\cap E$ is a $2$-group. 

By \cite[Chapter II, Proposition 11.7]{mazur77}, the Atkin-Lehner involution $w_M$ acts as $-1$ on $\Sigma(M)$. Further, since $E$ has an odd analytic rank, it follows by looking at the functional equation for $L(E,s)$ that $w_M$ acts as $1$ on $E$. Therefore, $-1=1$ on $\Sigma(M)\cap E$ meaning that $\Sigma(M)\cap E$ must be a $2$-group. 
\end{proof}

\Cref{thmling} is not enough to prove that $\xi_{E,N}^\vee$ is always injective for strong Weil curves of odd analytic rank. However, the computational evidence of \Cref{optimalquotient} seems to indicate the possibility that the $2$-group admitted by \Cref{thmling} cannot actually occur. We, therefore, make the following conjecture.
\begin{conj}\label{conj:injective}
Let $E$ be a strong Weil curve over $\Q$ of conductor $M\mid N$. If $E$ has an odd analytic rank, then $\xi_{E,N}^\vee$ is injective. 
\end{conj}

All but one of the strong Weil curves $E$ considered in the proof of \Cref{optimalquotient} have analytic rank $1$, the exception being the curve $389.a1$ with analytic rank $2$. Therefore, if this conjecture turns out to be correct, in the first part of \Cref{optimalquotient}, Sage will only be needed to prove that $\xi_{E,N}$ has a trivial kernel for the elliptic curve $389.a1$.

Note that the above conjecture, if true, makes the determination of all positive rank $d$-elliptic $X_0(N)$ significantly easier. Since the above conjecture together with \Cref{pairingcomputation} implies the following:
\begin{cor}
Assume \Cref{conj:injective}. Let $E$ be a strong Weil curve over $\Q$ of odd analytic rank, conductor $M$, and parametrization $f: X_0(M) \to E$. If $N$ is a multiple of $M$ and $g: X_0(N) \to E'$ is a map with $E'$ isogenous to $E$, then $\deg f \mid \deg g$.
\end{cor}

This would give us a lower bound on $\deg g$, allowing us to consider significantly fewer elliptic curves $E$ in the determination of positive rank $d$-elliptic $X_0(N)$.

\section{Curves \texorpdfstring{$X_0(N)$}{X\_0(N)} with infinitely many quartic points}\label{infinitelyquarticsection}

In this section, we will prove that for levels $N$ listed in the \Cref{quarticthm} the curve $X_0(N)$ has infinitely many quartic points. When $X_0(N)$ already has infinitely many quadratic points (these $N$ are listed in \Cref{thmquadratic}), this is trivial. Now we consider the other cases.

We will use two methods for obtaining quartic points on a curve $C$ defined over a number field $k$. Both methods obtain quartic points as pullback via rational maps from $C$. The first method uses a degree $4$ morphism to a curve $C'$ with infinitely many rational points (recall that Faltings' theorem implies that the only such curves $C'$ are of genus $0$ or genus $1$ with positive $k$-rank), and the second method uses a degree $2$ morphism to a curve $C'$ with infinitely many quadratic points. The following proposition verifies these methods.

\begin{prop}\label{degreedmap}
    Let $k$ be a number field and let $d$ be a positive integer. Suppose $C$ and $C'$ are smooth, projective, and geometrically integral curves defined over $k$ and let $f:C\to C'$ be a morphism of degree $d'\mid d$ defined over $k$. If $C'$ has infinitely many points of degree $\frac{d}{d'}$ over $k$, then $C$ has infinitely many points of degree $\leq d$ over $k$.
\end{prop}

\begin{proof}
    Let $P$ be a point on $C'$ of degree $\frac{d}{d'}$ over $k$ and let $K\supset k$ be its field of definition. Then the preimage $f^{-1}(P)$ has size $\leq d'$. Let $Q\in C(\overline{\Q})$ be an element of $f^{-1}(P)$. For every automorphism $\sigma\in G_K$, where $G_K$ is an absolute Galois group over $K$, we have
    $$f(\sigma(Q))=\sigma(f(Q))=\sigma(P)=P.$$
    Therefore, $\sigma(Q)\in f^{-1}(P)$ for every $\sigma\in G_K$. This means that, since $\# f^{-1}(P)\leq d'$, $Q$ must be defined over some field $L$ such that $[L:K]\leq d'$, or equivalently $[L:k]\leq d$.
\end{proof}

As we can see, this pullback method gives points of degree $\leq d$ over $k$. Therefore, if there are infinitely many points of degree $\leq d-1$ on $C$, we cannot immediately conclude that $C$ has infinitely many points of degree $d$. This can be resolved, however, using Theorems 4.2 and 4.3 of \cite{BELOV} which tells us that, as soon as one of the points in the pullback has degree $d$, there will be infinitely many points of degree $d$ on $C$. We will use this proposition to find infinitely many quartic points $X_0(N)$ by taking $d=4$ and $d'=1$ or $2$. 

\begin{remark}\label{remarkcubicquadraticpoints}
    Interestingly, from Theorems \ref{thmquadratic} and \ref{thmcubic} it follows that if $X_0(N)$ has infinitely many cubic points, then $X_0(N)$ also has infinitely many quadratic points. This means that the curve $X_0(N)$ has infinitely many points of degree $\leq 4$ if and only if it has infinitely many points of degree $4$.
\end{remark}

\begin{prop}\label{infinitelyquartic1}
    The curve $X_0(N)$ has infinitely many quartic points for $$N\in\{34,45,54,64,81\}.$$
\end{prop}

\begin{proof}
    For each of these $N$, the quotient $X_0(N)/\left<w_N\right>$ is an elliptic curve and therefore has infinitely many quadratic points. Now we use \Cref{degreedmap} for the degree $2$ quotient map from $X_0(N)$ to $X_0(N)/\left<w_N\right>$.
\end{proof}

\begin{prop}\label{infinitelyquartic2}
    The curve $X_0(N)$ has infinitely many quartic points for 
    \begin{align*}
        N\in\{&38,42,44,51,52,55-58,60,62,63,66-70,72-75,77,78,80,85,87,88,\\
        &91,92,94-96,98,100,103,104,107,111,119,121,125,142,143,167,191\}.
    \end{align*}
\end{prop}

\begin{proof}
    For each of these $N$ the curve $X_0(N)$ has $\Q$-gonality equal to $4$ by \cite[Tables 1,2,3]{NajmanOrlic22}. Using \Cref{degreedmap} we now conclude that there are infinitely many points of degree $\leq4$ on $X_0(N)$ for these $N$. Therefore, these curves $X_0(N)$ have infinitely many quartic points by \Cref{remarkcubicquadraticpoints}.
\end{proof}

\begin{prop}\label{tetraell}
    The curve $X_0(N)$ has infinitely many quartic points for 
    $$N\in\{82,86,99,118,123,141,145,155,159\}.$$
\end{prop}

\begin{proof}
\begin{center}
\begin{tabular}{|c|c|}
\hline
$N$ & LMFDB label of $X_0^*(N)$\\
  \hline

  $82$ & \lmfdbec{82}{a}{2}\\
  $86$ & \lmfdbec{43}{a}{1}\\
  $99$ & \lmfdbec{99}{a}{2}\\
  $118$ & \lmfdbec{118}{a}{1}\\
  $123$ & \lmfdbec{123}{b}{1}\\
  $141$ & \lmfdbec{141}{d}{1}\\
  $145$ & \lmfdbec{145}{a}{1}\\
  $155$ & \lmfdbec{155}{c}{1}\\
  $159$ & \lmfdbec{53}{a}{1}\\
    
  \hline
\end{tabular}
\end{center}
    For each of these $N$ we can use the Magma function X0NQuotient() to prove that the quotient $X_0^*(N)$ is an elliptic curve with the LMFDB label as in the table above. LMFDB also tells us that this elliptic curve is of rank $1$ over $\Q$. Now we use \Cref{degreedmap} for the degree $4$ quotient map from $X_0(N)$ to $X_0^*(N)$.
\end{proof}

\begin{remark}\label{tetraellremark}
    The proof of \Cref{tetraell} applies to these levels as well:

\begin{center}
\begin{tabular}{|c|c|}
\hline
$N$ & LMFDB label of $X_0^*(N)$\\
  \hline

  $57$ & \lmfdbec{57}{a}{1}\\
  $58$ & \lmfdbec{58}{a}{1}\\
  $74$ & \lmfdbec{37}{a}{1}\\
  $77$ & \lmfdbec{77}{a}{1}\\
  $91$ & \lmfdbec{91}{a}{1}\\
  $111$ & \lmfdbec{37}{a}{1}\\
  $142$ & \lmfdbec{142}{a}{1}\\
  $143$ & \lmfdbec{143}{a}{1}\\
    
  \hline
\end{tabular}
\end{center}
Also, the curve $X_{\textup{ns}}^+(11)$ is an elliptic curve with LMFDB label \lmfdbec{121}{b}{2}. It has conductor $121$, modular degree $4$, and rank $1$ over $\Q$. Therefore, we have a degree $4$ rational morphism from $X_0(121)$ to a positive rank elliptic curve.

We list these cases here separately since they have already been solved in \Cref{infinitelyquartic2}, but we need a morphism to a positive rank elliptic curve for \Cref{tetraellthm}.
\end{remark}

\begin{prop}\label{tetraell128}
    The curve $X_0(128)$ has infinitely many quartic points.
\end{prop}

\begin{proof}
    The elliptic curve $y^2=x^3+x^2+x+1$ has conductor $128$, modular degree $4$, and rank $1$ over $\Q$. This curve has Cremona label 128a1 \cite{cremona} and LMFDB label\lmfdbec{128}{a}{2}. Now we use \Cref{degreedmap} for the degree $4$ morphism from $X_0(128)$ to this elliptic curve.
\end{proof}

\section{Curves \texorpdfstring{$X_0(N)$}{X\_0(N)} with finitely many quartic points}\label{finitelyquarticsection}
In this section, we will prove that for levels $N$ not listed in the \Cref{quarticthm} the curve $X_0(N)$ has only finitely many quartic points. The first step to do that is to reduce this problem to a finite problem by giving an upper bound for $N$ such that the curve $X_0(N)$ has infinitely many quartic points.

As we mentioned in the Introduction, Frey's result \cite{frey} gives us that any curve defined over $\Q$ with infinitely many quartic points must have $\Q$-gonality $\leq8$. Furthermore, the theorem of Abramovich \cite{abramovich} gives us the lower bound on the $\C$-gonality of any modular curve. In our case, we get $\textup{gon}_\C X_0(N)\geq\frac{7}{800}N$ which means that for $N>\frac{8\cdot800}{7}$ the curve $X_0(N)$ has only finitely many quartic points (here we used a trivial fact that $\textup{gon}_\Q C\geq\textup{gon}_\C C$ for any curve $C$ defined over $\Q$). However, this bound is impractical here.

When the genus of $C$ is high enough, the following theorem by Kadets and Vogt tells us that this pullback method is the only way to obtain points of a certain degree.

\begin{thm}[\cite{KadetsVogt}, Theorem 1.4]\label{kadetsvogt1.4}
    Suppose $X/k$ is a curve of genus $g$ and $\textup{a.irr}_k X=d$. Let $m:=\lceil d/2\rceil -1$ and let $\epsilon:=3d-1-6m<6$. Then one of the following holds:
    \begin{enumerate}[(1)]
        \item There exists a nonconstant morphism of curves $\phi:X\to Y$ of degree at least $2$ such that $d=\textup{a.irr}_k Y\cdot\textup{deg}\phi$.
        \item $g\leq\textup{max} \left(\frac{d(d-1)}{2}+1,3m(m-1)+m\epsilon\right)$.
    \end{enumerate}
\end{thm}

\begin{cor}\label{degree4morphismcor}
    Suppose $C/\Q$ is a curve of genus $g\geq8$ and $\textup{a.irr}_\Q C=4$. Then there exists a nonconstant morphism of degree $4$ from $C$ to $\mathbb{P}^1$ or an elliptic curve defined over $\Q$ with positive $\Q$-rank.
\end{cor}

\begin{proof}
    We compute $m=1$ and $\epsilon=5$. Therefore, case (2) of the previous theorem is impossible and we have a morphism $f:C\to Y$ of degree $2$ or $4$. 
    
    If the degree of $f$ is $2$, then we have $\textup{a.irr}_\Q Y=2$ and $Y$ is a double cover of $\mathbb{P}^1$ or an elliptic curve with a positive $\Q$-rank by \Cref{kadetsvogt1.3}. If the degree of $f$ is $4$, then we have $\textup{a.irr}_\Q Y=1$ and $Y$ is isomorphic to $\mathbb{P}^1$ or an elliptic curve with a positive $\Q$-rank by Faltings' theorem.
\end{proof}

This means that for levels $N$ such that the genus of the curve $X_0(N)$ is at least $8$ the existence of infinitely many quartic points is equivalent with the existence of a degree $4$ morphism to $\mathbb{P}^1$ or to an elliptic curve with a positive $\Q$-rank.

Since for all levels $N$ not listed in the \Cref{quarticthm} the curve $X_0(N)$ has $\Q$-gonality $>4$ (by \cite[Table 2]{NajmanOrlic22}) and since $g(X_0(N))>7$ for all $N>100$, \Cref{degree4morphismcor} gives us that any potential $X_0(N)$ with infinitely many quartic points must be tetraelliptic. Now we can get a much better bound for $N$ using Ogg's inequality.

\begin{prop}[\cite{HasegawaShimura_trig}, Lemma 3.1, original source \cite{Ogg74}]\label{oggineq1}
    For a prime $p\nmid N$, put
    $$L_p(N):=\frac{p-1}{12}\psi(N)+2^{\omega(N)},$$
    where $\psi(N)=N\prod_{q\mid N}(1+\frac{1}{q})$ and $\omega(N)$ is the number of distinct prime divisors of $N$. Then
    $$\#\title{X}_0(N)(\F_{p^2})\geq L_p(N).$$
\end{prop}
\begin{cor}\label{oggineq2}
    If the curve $X_0(N)$ is tetraelliptic, then for every prime $p\nmid N$ we must have
    $$4(p+1)^2\geq L_p(N).$$
\end{cor}
\begin{proof}
    The proof can be found in \cite[Section 3]{JEON2018}.
\end{proof}

Now, applying \Cref{oggineq2} in the same way as in Lemma 3.2 of \cite{HasegawaShimura_trig}, we get

\begin{cor}\label{tetraellipticcor}
    The curve $X_0(N)$ is not tetraelliptic for all $N\geq402$ and 
    \begin{align*}
        N\in\{&154,174,190,198,202,204,212,222,224,228,231,232,234,236,244,246,\\
        &248,256,258,260,262,270,272,273,276,279,282,284-287,290,296,301,\\
        &303-306,308,310,312,316,318,320-322,324-328,330,332-336,\\
        &338-340,342,344-346,348,350-352,354-358,360,362-366,\\
        &368-372,374-378,380-382,384-388,390-396,398-400\}.
    \end{align*}
\end{cor}
This means that we only need to check a reasonably small number of levels $N$ for tetraellipticity. First, though, we separately solve the cases when $g(X_0(N))\leq7$ and we cannot use \Cref{kadetsvogt1.3}. The only $N$ not discussed already for which $g(X_0(N))\leq7$ is $N=97$.

\begin{prop}\label{prop97}
    The curve $X_0(97)$ has only finitely many quartic points.
\end{prop}

\begin{proof}
    Suppose that there are infinitely many quartic points on $X_0(97)$. Since the curve $X_0(97)$ has $\Q$-gonality $6$ by \cite[Table 2]{NajmanOrlic22} and genus $7$, we can apply \cite[Proposition 1.6]{JeonKimPark06} and get that the Jacobian $J_0(97)$ must contain an elliptic curve with a positive $\Q$-rank. However, up to isogeny, $J_0(97)$ only contains abelian varieties of dimension $3$ and $4$ and we get a contradiction.
\end{proof}

It is worth mentioning here that for levels $N$ in the following table there exist morphisms of degree $4$ to elliptic curves. For $N\neq109$, these morphisms are quotient maps which are, for $N$ divisible by $4$, composed with degree $2$ degeneracy maps from $X_0(N)$ to $X_0(\frac{N}{2})$. However, these elliptic curves are of rank $0$ over $\Q$ and therefore generate only finitely many quartic points.

\begin{table}[ht]
\centering
\begin{tabular}{|c|c|}
\hline
$N$ & $E$\\
  \hline

  $76$ & $X_0(38)/\left<w_{38}\right>$, $X_0(38)/\left<w_{19}\right>$ \\
  $105$ & $X_0(105)/\left<w_3,w_{35}\right>$ \\
  $108$ & $X_0(54)/\left<w_{54}\right>$, $X_0(54)/\left<w_{27}\right>$ \\
  $109$ & $y^2+xy=x^3-x^2-8x-7$ (LMFDB label\lmfdbec{109}{a}{1}) \\
  $110$ & $X_0(110)/\left<w_2,w_{55}\right>$ \\
  $112$ & $X_0(56)/\left<w_{56}\right>$, $X_0(56)/\left<w_7\right>$ \\
  $124$ & $X_0(62)/\left<w_{31}\right>$ \\
  $184$ & $X_0(92)/\left<w_{23}\right>$ \\
  $188$ & $X_0(94)/\left<w_{47}\right>$ \\
    
  \hline
\end{tabular}\\
\vspace{5mm}
\caption{Levels $N$ for which there exist morphisms of degree $4$ to an elliptic curve $E$ of rank $0$ over $\Q$}
\end{table}

Now we are ready to prove the two main theorems: \Cref{quarticthm} and \Cref{tetraellthm}.

\begin{proof}[Proof of \Cref{tetraellthm}]
The proofs of Propositions \ref{tetraell}, \ref{tetraell128} and \Cref{tetraellremark} tell us that $X_0(N)$ is positive rank tetraelliptic for all levels $N$ listed in \Cref{tetraellthm}. Now we prove that for the other $N$ the curve $X_0(N)$ is not positive rank tetraelliptic. We only need to consider $N < 408$ not already eliminated in \Cref{tetraellipticcor} for which there exists an elliptic curve of conductor $M \mid N$ of positive $\Q$-rank. Further, if $M=N$, then any morphism from $X_0(N)$ factors through the modular parametrization of a strong Weil curve in the corresponding isogeny class. However, the modular degree is strictly greater than $4$ in all such cases. Therefore, we may suppose $M<N$.

Let $E$ be a strong Weil curve of conductor $M \mid N$ and positive $\Q$-rank, $n$ the number of divisors of $N/M$, and $f: X_0(M) \to E$ its modular parametrization. Since $\xi_{E,N}: J_0(N) \to E^n$ for $N<408$ is an $E$-isogenous optimal quotient by \Cref{optimalquotient}, every map from $J_0(N)$ to $E$ uniquely factors through $E^n$. Therefore, we get that the maps $f \circ d_i$, where $d_i$ runs over the degeneracy maps $X_0(N) \to X_0(M)$, form a basis for $\Hom_\Q(J_0(N), E)\cong\Hom_\Q(E^n, E)\cong \Z^n$. \Cref{pairingcomputation} and \Cref{pairingcomputationmagma} allow us to compute the degree pairing on this basis. Now, the degree of a map $\sum_{i\mid N/M} x_i(f\circ d_i)$ is given by a positive definite quadratic form
$$\sum_{i\mid N/M}\sum_{j\mid N/M} x_ix_j\left<f\circ d_i,f\circ d_j\right>.$$
All $N$ and strong Weil curves $E$ which were considered in this proof are given in the \Cref{tab:main}.

Using the norm induced by this quadratic form, we can use the Fincke-Pohst algorithm for enumerating integer vectors of small norm \cite{finckepohst} to determine that there are no nonconstant elements of $\Hom_\Q(J_0(N), E)$ of degree $\leq 4$, and hence no elements of $\Hom_\Q(X_0(N),E)$ of degree $\leq 4$. So this proves the statement for strong Weil curves.

If $E$ is not a strong Weil curve, let $E'$ be a strong Weil curve in the isogeny class of $E$. Then by the $E$-isogenous optimality of $\xi_{E',N}: J_0(N) \to (E')^n$ we have that any $g \in \Hom_\Q(J_0(N), E)$ factors as $h\circ\xi_{E',N}$ for some $h\in\Hom_\Q((E')^n,E)$. 

Also, the map 
$$\pi:\Hom_\Q((E')^n,E)\to\Hom_\Q(E',E)^n, \ \pi(f)=(f\restriction_{E'_1},\ldots,f\restriction_{E'_n}),$$
where $E'_i$ is the $i$-th component of $(E')^n$, is an isomorphism with an inverse map 
$$\pi^{-1}((f_1,\ldots,f_n))(x_1,\ldots,x_n)=f_1(x_1)+\ldots+f_n(x_n).$$
Furthermore, we have that $\Hom_\Q(E',E)$ is a free $\Hom_\Q(E',E')(\cong \Z)$-module of rank $1$, generated by a single element $g_2$. In particular, any $f\in\Hom_\Q(E',E)$ can be written as $g_2\circ [m]$ for some $m\in\Z$.

Therefore, we have $\pi(h)=(f_1,\ldots,f_n)=g_2\circ ([m_1],\ldots,[m_n])$ and $h(x_1,\ldots,x_n)=g_2(m_1x_1+\ldots+m_nx_n)$. This means that $h=g_2\circ m$ for some $m\in\Hom_\Q((E')^n,E')$. Returning back to our $g\in\Hom_\Q(J_0(N),E)$, we see that it factors as $g_2\circ m \circ \xi_{E',N}$. It follows that
$$\deg g=\deg g_2\cdot \deg(m\circ \xi_{E',N})\geq\deg(m\circ\xi_{E',N})>4$$
since $E'$ is a strong Weil curve and $m\circ\xi_{E',N}$ is a rational map.
\end{proof}

In most cases, especially when we have only $2$ degeneracy maps, we do not actually need the Fincke-Pohst algorithm to prove that there are no nonconstant elements of $\Hom_\Q(J_0(N), E)$ of degree $\leq 4$. We show several examples where we prove that with elementary methods.

\begin{ex}
    We take $N=122$. There exist two elliptic curves $E$ of positive $\Q$-rank and conductor $\textup{cond}(E)\mid N$. One of them has conductor equal to $N$ and modular degree $8$ and can therefore be eliminated. The other one is $E=X_0^+(61)$. Its modular parametrization $f$ is the quotient map $X_0(61)\to X_0^+(61)$.

    By the proof of \Cref{tetraellthm}, the basis for $\textup{Hom}_\Q(J_0(122),E)$ is $\{f\circ d_1, f\circ d_2\}$ and both of these maps have degree $2\cdot3=6$. Further, by \Cref{pairingcomputation}, we have
    $$\left<f\circ d_1, f\circ d_2\right>=[a_2\cdot1\cdot 2]=[-2].$$
    This means that any map $J_0(122)\to E$ must have degree equal to 
    $$6x^2-4xy+6y^2$$ for some integers $x,y$. It remains to prove that this expression can never be equal to $4$.
    
    Let us suppose the contrary. If both $x$ and $y$ are not $0$, then $6x^2-4xy+6y^2=4x^2+2(x-y)^2+4y^2\geq8$. Therefore, we may without loss of generality set $y=0$. However, the expression now becomes $6x^2$ which cannot be equal to $4$, contradiction.
\end{ex}

\begin{ex}
    We take $N=129$. There exist two elliptic curves $E$ of positive $\Q$-rank and conductor $\textup{cond}(E)\mid N$. One of them has conductor equal to $N$ and modular degree $8$ and can therefore be eliminated. The other one is $E=X_0^+(43)$. Its modular parametrization $f$ is the quotient map $X_0(43)\to X_0^+(43)$.

    By the proof of \Cref{tetraellthm}, the basis for $\textup{Hom}_\Q(J_0(129),E)$ is $\{f\circ d_1, f\circ d_3\}$ and both of these maps have degree $2\cdot4=8$. Further, by \Cref{pairingcomputation}, we have
    $$\left<f\circ d_1, f\circ d_3\right>=[a_3\cdot1\cdot 2]=[-4].$$
    This means that any map $J_0(129)\to E$ must have degree equal to 
    $$8x^2-8xy+8y^2$$
    for some integers $x,y$. This expression is divisible by $8$ and cannot therefore be equal to $4$.
\end{ex}

\begin{ex}
    We take $N=148$. There exist two elliptic curves $E$ of positive $\Q$-rank and conductor $\textup{cond}(E)\mid N$. One of them has conductor equal to $N$ and modular degree $12$ and can therefore be eliminated. The other one is $E=X_0^+(37)$. Its modular parametrization $f$ is the quotient map $X_0(37)\to X_0^+(37)$.

    In this case $N/M$ is not squarefree like in the previous two examples. However, $N/M$ and $M$ are coprime and we can still use \Cref{pairingcomputation}. 

    By the proof of \Cref{tetraellthm}, the basis for $\textup{Hom}_\Q(J_0(148),E)$ is $\{f\circ d_1, f\circ d_2, f\circ d_4\}$ and these maps have degree $2\cdot6=12$. Further, by \Cref{pairingcomputation}, we have 
    $$\left<f\circ d_1, f\circ d_2\right>=[a_2\cdot2\cdot2]=[-8],$$
    $$\left<f\circ d_1, f\circ d_4\right>=[(a_4-a_1))\cdot1\cdot 2]=[2],$$
    $$\left<f\circ d_2, f\circ d_4\right>=[a_2\cdot2\cdot 2]=[-8].$$
    This means that any map $J_0(148)\to E$ must have degree equal to $$12x^2+12y^2+12z^2-16xy+4xz-16yz$$
    for some integers $x,y,z$. This expression is equal to 
    $$2(x+z-2y)^2+2(2x-y)^2+2(2z-y)^2+2x^2+2z^2.$$
    
    Let us suppose that it is equal to $4$ for some $x,y,z$. If both $x$ and $z$ are not $0$, then $2x^2+2z^2\geq4$ and the other terms must be equal to $0$. This would mean that $x+z-2y=2x-y=2z-y=0$. We can easily check that this is impossible. 
    
    Therefore, we may without loss of generality set $z=0$. The expression now becomes $12x^2-16xy+12y^2=8(x-y)^2+4x^2+4y^2$. As before, we see that $x$ or $y$ must be $0$ (otherwise $4x^2+4y^2\geq8$) and that $x=y$ (otherwise $8(x-y)^2\geq8$). This means that $x=y=z=0$ and we get a contradiction.
\end{ex}

\begin{proof}[Proof of \Cref{quarticthm}]
    The results in \Cref{infinitelyquarticsection} give us the cases when $X_0(N)$ has infinitely many quartic points and \Cref{prop97} tells us that the curve $X_0(97)$ has only finitely many quartic points. 
    
    For the other levels $N$, we have $g(X_0(N))\geq8$, $\textup{a.irr}_\Q(X_0(N))>3$, $\textup{gon}_\Q(X_0(N))>4$ by \cite[Tables 1,2,3]{NajmanOrlic22}, and that $X_0(N)$ is not positive rank tetraelliptic over $\Q$ by \Cref{tetraellthm}. Therefore, \Cref{degree4morphismcor} tells us that $X_0(N)$ has only finitely many quartic points for these levels $N$.
\end{proof}

\clearpage
\begin{table}[ht]
\centering
\begin{tabular}{|c|c|c|c|}
  \hline
  $N$ & $E$ & Modular Degree & Quadratic Form\\
  \hline

  $106$ & 53.a1 & $2$ & $6x^2-4xy+6y^2$\\  
  $114$ & 57.a1 & $4$ & $12x^2-16xy+12y^2$\\
  $116$ & 58.a1 & $4$ & $8x^2-8xy+8y^2$\\
  $122$ & 61.a1 & $2$ & $6x^2-4xy+6y^2$\\
  $129$ & 43.a1 & $2$ & $8x^2-8xy+8y^2$\\
  $130$ & 65.a1 & $2$ & $6x^2-4xy+6y^2$\\
  $148$ & 37.a1 & $2$ & $12x^2+12y^2+12z^2-16xy+4xz-16yz$\\
  $158$ & 79.a1 & $2$ & $6x^2-4xy+6y^2$\\
  $164$ & 82.a1 & $4$ & $8x^2-8xy+8y^2$\\
  $166$ & 83.a1 & $2$ & $6x^2-4xy+6y^2$\\
  $171$ & 57.a1 & $4$ & $12x^2-8xy+12y^2$\\ 
  $172$ & 43.a1 & $2$ & $12x^2+12y^2+12z^2-16xy+4xz-16yz$\\
  $176$ & 88.a1 & $8$ & $16x^2+16y^2$\\
  $178$ & 89.a1 & $2$ & $6x^2-4xy+6y^2$\\
  $182$ & 91.a1 & $4$ & $12x^2-16xy+12y^2$\\
        & 91.b2 & $4$ & $12x^2+12y^2$\\
  $183$ & 61.a1 & $2$ & $8x^2-8xy+8y^2$\\
  $184$ & 92.a1 & $6$ & $12x^2+12y^2$\\
  $185$ & 37.a1 & $2$ & $12x^2-8xy+12y^2$\\
  $195$ & 65.a1 & $2$ & $8x^2-8xy+8y^2$\\
  $215$ & 43.a1 & $2$ & $12x^2-16xy+12y^2$\\
  $237$ & 79.a1 & $2$ & $8x^2-4xy+8y^2$\\
  $242$ & 121.b2 & $4$ & $12x^2+12y^2$\\
  $249$ & 83.a1 & $2$ & $8x^2-4xy+8y^2$\\
  $259$ & 37.a1 & $2$ & $16x^2-4xy+16y^2$\\
  $264$ & 88.a1 & $8$ & $32x^2-48xy+32y^2$\\
  $265$ & 53.a1 & $2$ & $12x^2+12y^2$\\
  $267$ & 89.a1 & $2$ & $8x^2-4xy+8y^2$\\
  $297$ & 99.a2 & $4$ & $12x^2+12y^2$\\
    \hline
\end{tabular}\\
\vspace{5mm}
\caption{Levels $N$ and strong Weil curves $E$ (given by their LMFDB labels) considered in the proof of \Cref{tetraellthm}}
\label{tab:main}
\end{table}

\bibliographystyle{siam}
\bibliography{bibliography1}

\end{document}